\documentclass[amsppt,11pt]{amsart}
\usepackage{amsmath}
\usepackage{amsfonts}
\usepackage{amssymb}
\usepackage{amscd}

\def\noi{\noindent}

\def\what{\widehat}
\def\wtit{\widetilde}

\newtheorem{definition}{Definition}[section]

\newtheorem{theorem}{Theorem}[section]
\newtheorem{lemma}{Lemma}[section]
\newtheorem{corollary}{Corollary}[section]

\def\adots{\mathinner{\mkern2mu\raise0pt\hbox{.}  
\mkern2mu\raise4pt\hbox{.}\mkern1mu
\raise7pt\vbox{\kern7pt\hbox{.}}\mkern1mu}}

\begin{document}

\title[Quadratic unipotent blocks in general linear, unitary and symplectic
groups] {Quadratic unipotent blocks in general linear, unitary and
symplectic groups}
\author{Bhama Srinivasan}
  \address{Department of Mathematics, Statistics, and Computer Science (MC 249)\\
           University of Illinois at Chicago\\
           851 South Morgan Street\\
           Chicago, IL  60607-7045}
  \email{srinivas@uic.edu}

\begin{abstract}
An irreducible ordinary character of a finite reductive group is called
quadratic unipotent if it corresponds under Jordan decomposition
to a semisimple element $s$ in a dual group such that $s^2=1$.
We prove that there is a bijection between, on the one hand the
set of quadratic unipotent characters of $GL(n,q)$ or $U(n,q)$
for all $n \geq 0$ and on the other hand, the
set of quadratic unipotent characters of $Sp(2n,q)$ for all $n \geq 0$.
We then extend this correspondence to $\ell$-blocks for certain $\ell$
not dividing $q$.
\end{abstract}

\maketitle

\centerline{To Toshiaki Shoji}

\noindent
\section{Introduction}

\bigbreak\noindent Let ${\mathbb G}$ be a connected, reductive
algebraic group defined over ${\Bbb F}_q$ and $G$ the finite
reductive group of ${\mathbb F}_q$-rational points of ${\mathbb G}$.
The irreducible characters of $G$ are divided into rational Lusztig
series ${\mathcal E}(G,(s))$ where $(s)$ is a semisimple conjugacy
class in a dual group $G^*$ of $G$. Let $\ell$ be a prime not
dividing $q$. Each $\ell$-block of $G$ also determines a conjugacy
class $(s)$ in $G^*$, where now $s$ is an $\ell'$-semisimple
element. The block is said to be isolated if $C_{{\mathbb G}^*}(s)$
is not contained in a proper Levi subgroup of ${\mathbb G}^*$. If a
block is not isolated, the characters in the block in ${\mathcal E}(G,(s))$
can be obtained by Lusztig induction from a Levi subgroup of $G$.
Thus it is important to study the isolated blocks of
$G$. A description of the characters in isolated
blocks of classical groups when $\ell$ and $q$ are odd and $q$ is
large was given in \cite{BS} and \cite{BS2}.

\noi On the other hand, the notion of a perfect isometry between blocks
with abelian defect groups of two finite groups was introduced by
M.Brou\'e \cite{MB}. This leads to a comparison between an
$\ell$-block $B$ of a finite group $G$ and an $\ell$-block $b$ of a
group $H$. If there is a perfect isometry between $B$ and $b$,
certain invariants of the blocks are preserved. Often $H$ is a
``local subgroup" of $G$, for example the normalizer of a defect
group of $B$. In other situations $G$ and $H$ are finite groups of
the same type, e.g. symmetric groups , general linear groups or
unitary groups. (In fact there is a stronger result, i.e. the
abelian defect group conjecture, for symmetric groups and general
linear groups; see \cite{CR}.)

\noi In this paper we study quadratic unipotent characters, i.e.
characters in Lusztig series with $s^2=1$, and quadratic unipotent
blocks, i.e. blocks which contain quadratic unipotent characters,
 of general linear, unitary and symplectic groups. Here
we assume that $q$ and $\ell$ are odd. These blocks include
unipotent blocks and are isolated blocks for the symplectic group.
We first show that there is a natural bijection between the
quadratic unipotent characters of $GL(n,q)$ or $U(n,q)$ for all $n$ and the
quadratic unipotent characters of symplectic groups $Sp(2n,q)$ for all $n$.
Let $e$ be the order  of $q$ mod $\ell$.
If $B$ is a quadratic unipotent block of $GL(n,q)$ with $e$ even or
of $U(n,q)$  with $e$  odd or
$e \equiv 0 ({\rm mod} 4)$ we show that
there is a perfect isometry between $B$ and a quadratic unipotent
block $b$ of a symplectic group $Sp(2m,q)$. This kind of connection
between groups of type $A$ and $C$ appears to be new.

\noi Our main tool is the combinatorics of partitions and symbols related
to the blocks of general linear and symplectic groups. In particular
our work is inspired by a paper of Waldspurger \cite{W2}; a map
which is defined there between two combinatorial configurations can
be used to set up correspondences between blocks as above.

\noi The paper is organized as follows. In Section 2 we describe the construction
and parametrization of quadratic unipotent characters in $GL(n,q)$, $U(n,q)$
and $Sp(2n,q)$. Our main theorem, Theorem \ref{2D}, gives a bijection between the
sets of quadratic unipotent characters in $GL(n,q)$ or $U(n,q)$ for all $n \geq 0$
and the corresponding sets in $Sp(2n,q)$ for all $n \geq 0$. In Section 3 we
parameterize quadratic unipotent blocks with $e$ as above for these groups,
and in Section 4 we prove correspondences between blocks of $GL(n,q)$ or $U(n,q)$
for all $n \geq 0$ and blocks of $Sp(2n,q)$ for all $n \geq 0$. In Section 5 we
construct perfect isometries between corresponding blocks, in the case of abelian
defect groups. Finally in Section 6 we give an alternative interpretation of the
above correspondences. For the groups $G= GL(n,q)$ or $G= U(n,q)$ and $H=Sp(2n,q)$,
we consider groups $G(s)$ and $H(s)$ constructed by Enguehard as dual groups to the
centralizers of a semisimple element $s$ with $s^2=1$ in groups dual to $G$ or $H$.
We then interpret our correspondences as between unipotent blocks of $G(s)$ and $H(s)$.

\noi Notation:  If $G$ is a finite group, ${\rm Irr}(G)$ is the set of
(complex) irreducible characters of $G$. The Weyl group of type
$B_n$ is denoted by $W_n$. The Grothendieck group of an abelian
category $\mathcal C$ is denoted by $K_0(\mathcal C)$.

\noindent \section{ Quadratic Unipotent Characters }

\noindent If $G$ is a finite reductive group the set ${\rm
Irr}(G)$ is partitioned into geometric series by Deligne-Lusztig
 theory, and further into rational series ${\mathcal E}(G,(s))$ where $s\in G^*$ is a
semisimple element (see \cite{CE}, 8.23). For the groups $G$ that we study we
assume throughout this paper that $q$ is odd and $\ell$ is an odd prime not dividing $q$.

\begin{definition} If $\chi \in {\mathcal
E}(G,(s))$ where $s$ satisfies $s^2=1$ we  say $\chi$ is a quadratic
unipotent character.
\end{definition}

\noi These characters were called square-unipotent in \cite{BS2}. In
particular we have the unipotent characters, where $s=1$. If
$G=Sp(2n, q)$ (resp. $SO^\pm(2n,q)$) then $G^*=SO(2n+1,q)$ (resp.
$SO^\pm(2n ,q) $), and if $G=GL(n,q)$ or $G=U(n,q)$ then $G=G^*$. Since
$q$ is odd, if $s^2=1$ where $s\in G^*$ we get quadratic unipotent
characters in ${\mathcal E}(G,(s))$.

\noindent Let $G_n=GL(n,q)$ or $U(n,q)$.
The unipotent characters of $G_n$ are parameterized by partitions of
$n$. More generally, quadratic unipotent characters of $GL(n,q)$ have been
explicitly constructed by Waldspurger \cite{W2}. We generalize his
construction also to $U(n,q)$ below.

\noi  Let $({\mu}_1, {\mu}_2)$ be a pair of partitions where
${\mu}_i$ is a partition of $n_i,\ i=1,2$, with $n_1+n_2=n$. Let
$L=G_{n_1} \times G_{n_2}$ be a Levi subgroup of $G_n$, where
$G_{n_i}$  is a general linear or a unitary group according as
$G_n=GL(n,q)$ or $U(n,q)$ . Let $\mathcal E$ be the unique linear
character of $G_{n_2}$ of order $2$ and let ${\chi}_{\mu_i}$
be the unipotent character of $G_{n_i}$  corresponding to the partition
$\mu_i $. Then the virtual  character
$R^{G_n}_L({\chi}_{\mu_1} \times {\mathcal E} {\chi}_{\mu_2})$
obtained by Lusztig induction from $L$ (which in fact is Harish-Chandra
induction when $G_n=GL(n,q)$) is a quadratic unipotent
character, up to sign. We denote it by $\chi_{({\mu}_1, \mu_2)}$.
All quadratic unipotent characters of $G_n$ are obtained this way,
and thus we have a parametrization of quadratic unipotent characters
by pairs $({\mu}_1, {\mu}_2)$ such that $|{\mu}_1| + |{\mu}_2|=n$.
(We note also that by abuse of notation we use the finite groups
when we write $R^{G_n}_L$.)

\noi An alternative description of the quadratic unipotent
characters of $G_n=GL(n,q)$ or $U(n,q)$ is given as follows. These
characters are precisely the constituents of $R^{G_n}_L(1 \times
{\mathcal E} \times \chi_{(\kappa_1,\kappa_2)})$, where $L$ is a
Levi subgroup of the form $T_1 \times T_2 \times G_{n_0}$, $T_1$
(resp. $T_2$) is a product of $N_1$ (resp. $N _2$) tori of order
$q^2-1$. Let $1$ be the trivial character(of $T_1$ and $\mathcal E$ the
product of the characters of order 2 on each component of $T_2$. The character
${\chi}_{(\kappa_1,\kappa_2)}$ is a $2$-cuspidal character of
$G_{n_0}$, i.e. ${\kappa_1}$ and ${\kappa_2}$ are $2$-cores. We
note that in this case, by the work of Lusztig \cite {L1}
the $R^{G_n}_L$ map
is Harish-Chandra induction for  $U(n,q)$. The endomorphism algebra
of the induced representation is isomorphic to a Hecke algebra of
type $W_{N_1} \times W_{N_2}$.

\noi Let $H_n=Sp(2n,q)$, $q$ odd. We have a similar
description of quadratic unipotent characters of $H_n$, as given by
Lusztig \cite {L1} and Waldspurger (\cite {W1}, 4.9). The
characters are constituents of $R^{H_n}_K(1 \times {\mathcal E}
\times \chi)$, where $K$ is a Levi subgroup of the form $T_1 \times
T_2 \times H_{n_0}$, $T_1$ (resp. $T_2$) is a product of $N_1$ (resp.
$N _2$) tori of order $q-1$. Let $1$ be the trivial character of $T_1$
and $\mathcal E$ the product of the characters of order 2 on each component
of $T_2$. The character $\chi$ is a cuspidal quadratic unipotent
character of $H_{n_0}$ and the $R^{H_n}_K$ map is Harish-Chandra
induction. The endomorphism algebra of the induced representation
is again isomorphic to a Hecke algebra of type $W_{N_1} \times
W_{N_2}$.

\noi We now describe the combinatorics of symbols needed to
parameterize  the quadratic unipotent characters of $H_n$.
 By the work of Lusztig \cite{L1} the unipotent characters of
classical groups are parameterized by equivalence classes of
symbols.  We refer to (\cite{C}, p.465),
(\cite{BMM}, p.~48) for a description of the symbols associated with
unipotent characters of $Sp(2n,q)$, including definitions of the
equivalence relations on symbols and the rank and defect of a
symbol.

\noi We denote a symbol by $\Lambda=(S,T)$ where
$S,T\subseteq {\Bbb N}$. If $\Lambda$ is unordered, it is regarded
as the same as $(T,S)$ and also the same as the symbol obtained by a
shift operation from itself ( \cite{C}, p.~375). The defect of
$\Lambda$ is $|S|-|T|$. We also need to consider ordered symbols to
parameterize unipotent characters of $O^\pm(2n,q)$, which were
described by Waldspurger.

\noi We then have:
\begin{itemize}
\item
 The unipotent characters of $Sp(2n,q)$ are in
bijection with unordered symbols of rank $n$ and odd defect.
\item The unipotent characters of $O^+(2n,q)$ are in
bijection with ordered symbols of rank $n$ and defect $\equiv 0 \
({\rm mod}\ 4)$
\item The unipotent characters of $O^-(2n,q)$ are in
bijection with ordered symbols of rank $n$ and defect $\equiv 2 \
({\rm mod}\ 4)$
\item The irreducible characters of $W_n$ are in bijection
with  unordered symbols of rank $n$ and  defect 1.
\end{itemize}

\noi The operations of ``adding an $a$-hook'' to and
``deleting an $a$-hook" from a partition, and the concept of an
``$a$-core" of  a partition are well-known. Similarly we have
operations of ``adding an $a$-hook or an $a$-cohook'' and ``deleting
an $a$-hook or $a$-cohook" to a symbol $\Lambda$. They can be
described as follows ( \cite{O}, p.-226). Let $\Lambda=(S,T)$. We
say a symbol $\Lambda'$ is obtained from $\Lambda$ by adding an
$a$-hook if it is obtained by deleting a member $x$ of $S$ (or $T$)
and inserting $x+a$ in $S$ (or $T$). We say $\Lambda'$ is obtained
from $\Lambda$ by adding an $a$-cohook if it is obtained from
$\Lambda$ by deleting a member $x$ of $S$ (or $T$) and inserting
$x+a$ in $T$ (or $S$).

\noi We follow the notation of \cite{W1} below.
We define a map $\sigma$ on
ordered symbols by $\sigma(S,T)=(T,S)$.  Let $\wtit S_{n,d}$ be the
set of ordered symbols of rank $n$ and defect $d$, and let
$S_{n,d}=\wtit S_{n,d}\cup\wtit S_{n,-d}$, modulo the relation
$\Lambda \sim\sigma(\Lambda)$.

Let
\begin{eqnarray*}
S_{n,odd} &=& \bigcup_{\substack{ d\in{\Bbb N}\\ d \text{ odd} }}
S_{n,d}, \
\wtit S_{n,\,even}=\bigcup\limits_{\substack{ d\in{\Bbb Z}\\ d \text{ even} }} \wtit S_{n,d},\\
S\wtit S_{n, \text{ mix}} &=& \bigcup\limits_{n_1+n_2=n}( S_{n_1,
\text{ odd}} \times \wtit S_{n_2, \text{ even}}).
\end{eqnarray*}

 \noi {\bf Remark.}  We have taken the
liberty of replacing ``pair'' by ``even'' and ``imp'' by ``odd'' in
\cite{W1}.

\noi By the work of Lusztig \cite{L1} and Waldspurger \cite{W1} we
have a parametrization of the quadratic unipotent characters of
$H_n$ by $S\wtit S_{n, \text{ mix}}$ which generalizes that of the
unipotent characters, given above. This will be clarified in
Lemma \ref{2B} below.

\noi
We note that if $\rho \in {\rm Irr} (W_n)$ there is a symbol of defect $1$
corresponding to $\rho$ (\cite{C}, p.375). By abuse of
notation we will sometimes refer to "the core (or cocore) of $\rho$",
to mean the core (or cocore) of the symbol. The characters in ${\rm Irr} (W_n)$
are also parameterized by pairs of partitions $(\lambda_1, \lambda_2)$ with
$\lambda_1 + \lambda_2=n$, and this will be used in the lemma below.

\noindent We now give the parametrization of the quadratic
unipotent characters of $G_n$ and $H_n$ which we will use in our
description of blocks. We remark that the parametrization by
4-tuples in the case of $G_n$, rather than by pairs of partitions
is crucial for our results.

\noi \begin{lemma}\label{2A} The quadratic unipotent
characters of $G_n$ can be parameterized by 4-tuples $(m_1, m_2,
{\rho}_1, {\rho}_2)$ such that
$$m_1(m_1+1)/2 + m_2(m_2+1)/2 +2N_1+2N_2=n,$$ where $m_1,m_2 \in
{\Bbb N}$ and ${\rho}_i \in {\rm Irr}(W_{N_i})$, $i=1,2$.
\end{lemma}

\begin {proof} The quadratic unipotent characters of $G_n$ are
parameterized by pairs of partitions $({\mu}_1, {\mu}_2)$ such that
$|{\mu}_1| + |{\mu}_2|=n$. A combinatorial proof that
we may parameterize these
characters of $G_n$ by 4-tuples $(m_1, m_2, {\rho}_1, {\rho}_2)$ as
above is given in (\cite{W2}, p.361). The characters
${\rm Irr}( W_{N_i})$ are also parameterized by pairs of partitions, and so
 we can regard each ${\rho}_i$ as corresponding to a pair of partitions.
Then $\chi_{({\mu}_1, \mu_2)}$ is parameterized by $(m_1, m_2,
{\rho}_1, {\rho}_2)$ where the 2-core of ${\mu}_i$ is
$\lbrace m_i, m_{i-1}, \ldots , 2, 1\rbrace $ and the 2-quotient of
${\mu}_i$ is ${\rho}_i \in {\rm Irr}( W_{N_i})$, $i=1,2$. Here
${\rho}_i$ corresponds to a pair of partitions, as mentioned above.

\noi We note that the character parameterized by $(m_1, m_2, -,
-)$ is 2-cuspidal for $GL(n,q)$. In the case of $U(n,q)$ the description given
above also shows that we can regard this parametrization as coming
from Harish-Chandra induction from a suitable Levi subgroup $L$,
with $(m_1, m_2, -, -)$ the parameters for a cuspidal quadratic
unipotent character of a possibly smaller unitary group $U(n_0,q)$
and with $({\rho}_1, {\rho}_2)$  the character of a  product of two
Hecke algebras of type $B$ corresponding to $W_{N_1} \times
W_{N_2}$. This gives another proof of the parametrization by the
4-tuples as above for $U(n,q)$, and hence for $GL(n,q)$.
\end{proof}

\noi {bf Remark.} For an explanation of the connection between the two parameterizations
of unipotent characters of $U(n,q)$ see also (\cite{FS3}, p.224).

\noi \begin{lemma} \label{2B} The quadratic unipotent characters of
$H_n$ can be parameterized by pairs of symbols $(\Lambda_1,\Lambda_2)$
and by 4-tuples $(h_1, h_2, {\rho}_1,
{\rho}_2)$ such that $h_1(h_1+1) + h_2^2 +N_1+N_2=n$, where $h_1 \in
{\Bbb N}$, $h_2 \in {\bf Z}$ and ${\rho}_i \in {\rm Irr}\
W_{N_i}$, $i=1,2$
\end{lemma}

\begin{proof} As in the case of $U(n,q)$
this is done by Harish-Chandra induction of cuspidal
quadratic-unipotent characters from a suitable Levi subgroup $K$
(\cite{W1}, 4.9-4.11). The endomorphism algebra of
the induced representation is again isomorphic to a Hecke algebra of type
$W_{N_1} \times W_{N_2}$. Hence the set of quadratic unipotent
characters of $H_n$ is parameterized by 4-tuples $(h_1, h_2,
{\rho}_1, {\rho}_2)$, where the cuspidal character is parameterized
by $(h_1, h_2, -, -)$.  Then (\cite{W1}, 2.21, 4.10) the pair
$(h_1,\rho_1)$ corresponds to a symbol $\Lambda_1 \in S_{h_1 +
h_1^2+N_1,odd}$ and the pair $(h_2,\rho_2)$ corresponds to a symbol
$\Lambda_2 \in \wtit S_{ h_2^2+N_2,even}$. Thus there is a pair
$(\Lambda_1,\Lambda_2)\in S\wtit S_{n, mix}$ corresponding to the
4-tuple $(h_1, h_2, {\rho}_1, {\rho}_2)$, and there is a bijection
of $S\wtit S_{n, mix}$ with the set of quadratic unipotent
characters of $H_n$.

\noi We note here the connection between the symbols $\Lambda_1,\Lambda_2$
and the symbols corresponding to $\rho_1, \rho_2$.  Suppose the symbol corresponding
to $\rho_1$ is $(S,T)$ where $|S|=|T|+1$.  Then the symbol corresponding
to $\Lambda_1$ is $(S',T)$ where, if $ 2h_1+1=d$, $S'=\lbrace [0, d-2]\cup ( S+d-1) \rbrace$
(\cite{L1}, 3.2). The formula for $\rho_2$ and $\Lambda_2$ is similar.
\end{proof}

\noi The quadratic unipotent character parameterized by
$(\Lambda_1,\Lambda_2)$ is denoted by
$\chi_{(\Lambda_1,\Lambda_2)}$.

\noi {\bf Remark.} We note that since $(\Lambda_1,\Lambda_2)\in
S\wtit S_{n, mix}$ the character $\chi_{(\Lambda_1,\Lambda_2)}$ is in ${\mathcal
E}(H_n, (s))$ where the number of  eigenvalues of $s$ equal to $1$
(resp. $-1$) in the natural representation of the dual group
$SO(2n+1)$ is $2 \ {\rm rank} (\Lambda_1)+1$ (resp. $2 \ {\rm rank}
(\Lambda_2)$). The pair $(\Lambda_1,\Lambda_2)$ parameterizes a unipotent character of
the centralizer of $s$ in a group dual to $H_n$, and thus we have the Jordan
decomposition of $\chi_{(\Lambda_1,\Lambda_2)}$. This will be used in Section 5.

\noi The following lemma is a first step towards connecting the quadratic
unipotent characters of the groups $G_n$ and the groups $H_n$.

\noindent\begin{lemma} \label {2C} (\cite{W2}, p.362).
There is a bijection between pairs $(m_1, m_2)$ such that
$m_1(m_1+1)/2 + m_2(m_2+1)/2=n$ and pairs $(h_1, h_2)$ such that
$h_1(h_1+1) + h_2^2=n$. This bijection is defined by $m_1=
sup(h_1+h_2, -h_1-h_2-1)$ and $m_2= sup(h_1-h_2, h_2-h_1-1)$.
\end{lemma}

\noi {\bf Remark.} Note that if $h_2$ is replaced by $-h_2$, $m_1$ and $m_2$
are interchanged in the above bijection.

\noi
This bijection then leads to the following result, which is crucial
to us. The proof is a straightforward extension of the above lemma.

\noindent \begin{theorem}\label {2D}  The map $(m_1, m_2, {\rho}_1,
{\rho}_2)\rightarrow (h_1, h_2, {\rho}_1, {\rho}_2)$, ${\rho}_i \in
{\rm Irr}\ W_{N_i}$, $i=1,2$  induces a bijection between the  set of
quadratic unipotent characters of $(G_n, n \geq 0) $,  and  the set of quadratic
unipotent characters of $(H_n, n \geq 0) $. Under this
bijection the character corresponding to $(m_1, m_2, {\rho}_1, {\rho}_2)$ of $G_n$ maps to
the character corresponding to $(h_1, h_2, {\rho}_1, {\rho}_2)$ of $H_m$
where  $m_1(m_1+1)/2 + m_2(m_2+1)/2 +
2N_1+2N_2=n$ and $h_1(h_1+1) + h_2^2+N_1+N_2=m$.
\end{theorem}

\noindent {\bf Example}. The group $Sp(4,q)$ has 23 quadratic
unipotent characters (and only 6 unipotent characters). Of these, 14
characters are in bijection with quadratic unipotent characters of
$GL(4,q)$, 8 with those of $GL(3,q)$ and 1 with that of $GL(2,q)$.
The latter is the unipotent cuspidal character $\theta_{10}$, which
is in bijection with the quadratic unipotent (not unipotent)
2-cuspidal character of $GL(2,q)$ parameterized by the pair of
partitions $(1,1)$ or by the $4$-tuple $(1, 1, -, -)$. Here
$m_1=m_2=1, h_1=1, h_2=N_1=N_2=0$.

\noindent {\bf Example}. The group $GL(4,q)$ has 20 quadratic
unipotent characters (and only 5 unipotent characters). Of these, 14
characters are in bijection with quadratic unipotent characters of
$Sp(4,q)$, 4 with those of $Sp(6,q)$ and 2 with those  of $Sp(8,q)$.
The latter are cuspidal quadratic unipotent characters of $Sp(8,q)$
corresponding to cuspidal quadratic unipotent characters of
$O^+(8,q)$ under Jordan decomposition. They are in bijection
with the quadratic unipotent
2-cuspidal characters of $GL(4,q)$ parameterized by the pair of
partitions $(21,1)$.  Here $m_1=2, m_2=1, h_2=2, h_1=N_1=N_2=0$,
or $(1,21)$  with $m_1=1, m_2=2, h_2=-2, h_1=N_1=N_2=0$.

\noi Theorem \ref{2D} can be restated as follows. Let $L_n$ (resp.
$L_n'$) be the category of quadratic unipotent characters of $G_n$
(resp. $H_n$).

\noindent \begin{theorem} \label {2E} There is an isomorphism
(isometry) between the
 groups $\oplus_{n \geq 0}\ K_0(L_n)$ and $\oplus_{n \geq 0}\ K_0(L_n')$
given by mapping the character parameterized by $(m_1, m_2,
{\rho}_1, {\rho}_2)$ to the character parameterized by $(h_1, h_2,
{\rho}_1, {\rho}_2)$, ${\rho}_i \in {\rm Irr}\ W_{N_i}$, $i=1,2$.
\end{theorem}

\section {Quadratic unipotent blocks}

\noi  The $\ell$ - blocks of $G_n$ and of the conformal symplectic
group $CSp(2n,q)$ were classified in \cite {FS1}, \cite {FS2}. We
define a quadratic unipotent block of $G_n$ or $H_n$ to be one which
contains quadratic unipotent characters. As a special case we have
the unipotent blocks, which have been studied by many authors (see e.g.
 \cite{CE}). The quadratic unipotent $\ell$ - blocks of $H_n$ were
classified in terms of cuspidal pairs in \cite {BS}. A description of the
characters in a quadratic unipotent block of $H_n$ was given in
\cite {BS2} if $q > 2n$ .

\noi
The following theorem describes these results.
Here and in the rest of the paper, $e$ is the order of $q$ mod
$\ell$ and $f$ the order of $q^2$ mod $\ell$. The character
$\mathcal E$ of the torus $T_2$ is the product of the characters
of order 2 on each component of $T_2$.

\noi \begin{theorem} \label{3A} (i) \cite {FS1}
Let $\ell$ divide
$q^f+1$ if  $G_n=GL(n,q)$ and let $\ell$ divide $q^f+1$, $f$ even, or
$q^f-1$, $f$ odd, if $G_n=U(n,q)$. Let $B$ be a quadratic unipotent
$\ell$-block of $G_n$. Then $B$ corresponds to a pair $({\lambda}_1,
{\lambda}_2)$ of partitions such that $|\lambda_1| + |\lambda_2|=n'$
and such that ${\lambda}_1$ and ${\lambda}_2$ are $2f$-cores, i.e. no
$2f$-hooks can be removed from them. The quadratic unipotent
characters in $B$ are of the form $\chi_{({\mu}_1, \mu_2)}$ where
${\lambda}_i$ is the $2f$-core of ${\mu}_i$ ($i=1,2$). These
characters are precisely the constituents of $R^{G_n}_L(1 \times
{\mathcal E} \times \chi_{(\lambda_1,\lambda_2)})$, where $L$ is a
Levi subgroup of the form $T_1 \times T_2 \times G_{n'}$, $T_1$
(resp. $T_2$) is a product of $M_1$ (resp. $M _2$) tori of order
$q^{2f}-1$, and $1$ (resp. $\mathcal E$) is the trivial character
(resp. character of order 2) of $T_1$ (resp. $T_2$). The character
${\chi}_{(\lambda_1,\lambda_2)})$ is in a block of defect $0$ of
$G_{n'}$.

\noi (ii) \cite{BS2} Let $q > 2m$. Let $b$ be a quadratic unipotent
$\ell$ - block, i.e. an isolated block of $H_m$ and let $\ell$
divide $q^f-1$, $f$ odd. Then $b$ corresponds to a pair of symbols
$({\pi}_1, {\pi}_2)$ where the ${\pi}_i$ are $f$-cores .
The quadratic unipotent characters in $b$ are of the form
$\chi_{({\Lambda}_1, \Lambda_2)}$ where ${\pi}_i$ is the $f$-core of
${\Lambda}_i$ ($i=1,2$). These characters are precisely the
constituents of $R^{H_m}_K(1 \times {\mathcal E} \times
{\chi}_{({\pi_1},{\pi_2})})$, where $K$ is a Levi subgroup of the
form $T_1 \times T_2 \times H_{m'}$, $T_1$ (resp. $T_2$) is a
product of $M_1$ (resp. $M_2$) tori of order $q^{f}-1$ and $1$
(resp. $\mathcal E$) is the trivial character (resp. character of
order 2) of $T_1$ (resp. $T_2$). The character
${\chi}_{({\pi_1},{\pi_2})}$ is in a block of defect $0$ of
$H_{m'}$.

\noi (iii) \cite{BS2} Let $q > 2m$. Let $b$ be a quadratic unipotent
$\ell$ - block, i.e. an isolated block of $H_m$ and let $\ell$
divide $q^f+1$. Then $b$ corresponds to a pair of symbols $({\pi}_1,
{\pi}_2)$ where the ${\pi}_i$ are $f$-cocores. The quadratic
unipotent characters in $b$ are of the form $\chi_{({\Lambda}_1,
\Lambda_2)}$ where ${\pi}_i$ is the $f$-cocore of ${\Lambda}_i$
($i=1,2$). These characters are precisely the constituents of
$R^{H_m}_K(1 \times {\mathcal E} \times
{\chi}_{({\pi_1},{\pi_2})})$, where $K$ is a Levi subgroup of the
form $T_1 \times T_2 \times H_{m'}$, $T_1$ (resp. $T_2$) is a
product of $M_1$ (resp. $M_2$) tori of order $q^{f}+ 1$ and $1$
(resp. $\mathcal E$) is the trivial character (resp. character of
order 2) of $T_1$ (resp. $T_2$). The character
${\chi}_{({\pi_1},{\pi_2})}$ is in a block of defect $0$ of
$H_{m'}$. $\hfill \square $
\end{theorem}

\noi  The following combinatorial
lemma due to Olsson (\cite{O}, p.235) and to Enguehard
(\cite{E2},5.7 ) will be used to connect blocks of types (ii) and
(iii) in the above theorem.

\noi \begin{lemma} \label{3B}
 Given a symbol $\Lambda$ of rank $n$ and a positive integer
$e$ one can define a symbol $\what{\Lambda}$, called the
$e$-twisting of $\Lambda$ in (\cite{O},p.235) such that there is a
bijection between $e$-cohooks in $\Lambda$ and $e$-hooks in
$\what{\Lambda}$. In particular if $\what{\Lambda}$ is an $e$-core,
i.e. has no $e$-hooks, then $\Lambda$ is an $e$-cocore, i.e. has no
$e$-cohooks.
\end{lemma}

\noi \begin{corollary} The operation of $e$-twisting is an
involution on the set of quadratic unipotent characters of
$Sp(2n,q)$.
\end{corollary}

\noi \begin{theorem} \label{3C}
 If $G_n=GL(n,q)$, let $e=2f$ be the order
of $q$ mod $\ell$, so that $\ell$ divides $q^f+1$. (We exclude the
case where $e$ is odd.)
 If $G_n=U(n,q)$ let again $e$ be the order
of $q$ mod $\ell$ and $f$ the order of $q^2$ mod $\ell$. Consider
the two cases: (i) $e=f$ is odd, $\ell$ divides $q^{2f}-1$ and
$q^f-1$, or (ii) $e=2f$ where $f$ is even, i.e. $e \equiv 0 \ ({\rm
mod}\ 4)$ and $\ell$ divides $q^f+1$. The case $e \equiv 2 \ ({\rm
mod}\ 4)$ is excluded. Then the quadratic unipotent blocks of $G_n$
are parameterized by 6-tuples $(m_1, m_2, {\sigma}_1, {\sigma}_2,
M_1,M_2 )$, where ${\sigma}_i \in {\rm Irr}\ W_{N'_i}$, $i=1,2$ with
$fM_1+N_1'=N_1, \ fM_2+N_2'=N_2$, $m_1(m_1+1)/2 + m_2(m_2+1)/2 +
2N_1+2N_2=n$. The  quadratic unipotent characters in a block
parameterized by $(m_1, m_2, {\sigma}_1, {\sigma}_2, M_1,M_2 )$ are
then parameterized by 4-tuples $(m_1, m_2, {\rho}_1, {\rho}_2)$ such
that $( {\rho}_1, {\rho}_2)$ have $({\sigma}_1, {\sigma}_2)$ as
$f$-cores.
\end{theorem}
\begin{proof} We use Theorem \ref{3A} and the construction
of quadratic unipotent characters. Let $B$ be a quadratic unipotent
$\ell$-block of $G_n$.
We have the following configurations, by our choice of $e$.
The block $B$ corresponds to an $e$-split Levi subgroup of $G_n$
which is a product of $M_1 + M_2$ tori of order $q^{2f}-1$ and $G_{n'}$.
Then $G_{n'}$ has a $2$-split Levi subgroup which is a product of
${N_1}' + {N_2}'$ tori of order $q^{2}-1$ and $G_{n_0}$, and finally
$G_{n}$ has a $2$-split Levi subgroup which is a product of
$N_1 + N_2$ tori of order $q^{2}-1$ and $G_{n_0}$.

\noi Then $B$ corresponds to a pair $({\lambda}_1,
{\lambda}_2)$ of  $2f$-cores which parameterize a block of defect $0$
of $G_{n'}$. Suppose the $2$-core
of $({\lambda}_1, {\lambda}_2)$ is $(\kappa_1, \kappa_2)$.
Then $(\kappa_1, \kappa_2)$ is
parameterized by a 4-tuple $(m_1,m_2, -, -)$, where $\kappa_i$ is
the partition $(m_i, m_i-1, \ldots 1)$ for $i=1,2$. Then the
$2f$-core $({\lambda}_1, {\lambda}_2)$ is parameterized by a 4-tuple
$(m_1, m_2, {\sigma}_1, {\sigma}_2)$, where  ${\sigma}_i \in {\rm
Irr}\ W_{N'_i}$, $i=1,2$, and $m_1(m_1+1)/2 + m_2(m_2+1)/2 +
2N_1'+2N_2'=n'$. Since $B$ is parameterized by the pair of the
$e$-split Levi subgroup and the character $({\lambda}_1, {\lambda}_2)$,
we get the parametrization of $B$ by the sextuple
$(m_1, m_2, {\sigma}_1, {\sigma}_2,
M_1,M_2 )$, where ${\sigma}_i \in {\rm Irr}\ W_{N'_i}$, $i=1,2$.

\noi Let $\chi_{(\mu_1, \mu_2)} \in B$. Now
${\lambda_1}$ and ${\lambda_2}$ are obtained from ${\mu}_1$ and
${\mu_2}$ respectively by removing $2f$-hooks. Removing a $2f$-hook
can be achieved by removing $f$ $2$-hooks. Thus all the $({\mu}_1,
\mu_2)$ parameterizing the quadratic unipotent characters in $B$
have the same $2$-core $(\kappa_1, \kappa_2)$. Then  all the 4-tuples parameterizing
the quadratic unipotent characters in $B$ have the form $(m_1, m_2,
{\rho}_1, {\rho}_2)$ such that $m_1(m_1+1)/2 + m_2(m_2+1)/2
+2N_1+2N_2=n$, where ${\rho}_i \in {\rm Irr}W_{N_i}$, $i=1,2$ . In
other words the pair $(m_1, m_2)$ is fixed for all the characters.
We then note (see Lemma \ref{2A}) that $({\sigma}_1, {\sigma}_2)$ are the $2$-quotients of the
partitions $({\lambda}_1, {\lambda}_2)$, and hence $\sigma_1$ and $\sigma_2$ are $f$-cores.
A count of the number of $2$-hooks removed from a pair of partitions
to reach the $2$-core gives
 $$ fM_1+N_1'=N_1, \ fM_2+N_2'=N_2.$$
 This gives the result.
\end{proof}

\noi The proof of the next proposition for the groups $H_m$ and the case
of $\ell$ dividing  $q^f-1$ is similar to the above.

\noi \begin{theorem} \label{3D} Let $q > 2m$.
Let $\ell$ divide $q^f-1$, $f$ odd.
 The quadratic unipotent blocks of $H_m$ are parameterized by
6-tuples $(h_1, h_2, {\sigma}_1, {\sigma}_2, M_1,M_2 )$, where
${\sigma}_i \in {\rm Irr}\ W_{N'_i}$, $i=1,2$ with $fM_1+N_1'=N_1, \
fM_2+N_2'=N_2$, $h_1(h_1+1) + h_2^2 +N_1+N_2=m$. Here the symbols
corresponding to $\sigma_1$ and ${\sigma}_2$ are $f$-cores.
 The quadratic unipotent characters in $b$  are
parameterized by  4-tuples of the form $(h_1, h_2, {\rho}_1,
{\rho}_2)$ where $({\rho}_1, {\rho}_2)$ have $(\sigma_1, \sigma_2)$
as  $f$-cores.
\end{theorem}
\noi \begin{proof}   Let $b$ be a
quadratic unipotent $\ell$-block of $H_m$ corresponding to a pair of symbols
$(\pi_1,\pi_2)$ which are $f$-cores, as in Theorem \ref{3A}.
 The $1$-core of $(\pi_1,\pi_2)$ is
parameterized by $(h_1, h_2, - , -)$ for some $h_1, h_2$
 and $(\pi_1,\pi_2)$ is parameterized
by $(h_1, h_2, {\sigma}_1, {\sigma}_2 )$, where ${\sigma}_i \in {\rm
Irr}\ W_{N'_i}$, $i=1,2$. To show that ${\sigma}_1$ is an
$f$-core, we can assume $\pi_1 =(S',T)$, and that the symbol
corresponding to ${\sigma}_1$ is $(S,T)$  as in
Lemma \ref{2B}. Using the description given there of the connection
between $S$ and $S'$ it is easy to see that removing an $f$-hook from
$(S',T)$ is equivalent to removing an $f$-hook from $(S,T)$.
Thus $(S',T)$ is an $f$-core if and only if $(S,T)$ is an $f$-core.

\noi Let $\chi_{(\Lambda_1, \Lambda_2)} \in
b$.  Now ${\pi}_1$ and ${\pi_2}$ are obtained from ${\Lambda_1}$ and
${\Lambda_2}$ respectively by removing $f$-hooks . Removing an
$f$-hook can be achieved by removing $f$ $1$-hooks. Thus all the
$({\Lambda}_1, \Lambda_2)$ parameterizing the quadratic unipotent
characters in $b$ have the same $1$-core which is the $1$-core of
$({\pi}_1, {\pi}_2)$.

\noi Furthermore all the 4-tuples parameterizing the quadratic
unipotent characters in $b$ have the form $(h_1, h_2, {\rho}_1,
{\rho}_2)$ such that $h_1(h_1+1) + h_2^2 +N_1+N_2=m$, where
${\rho}_i \in {\rm Irr} \ W_{N_i}$, $i=1,2$ . In other words the
pair $(h_1, h_2)$ is fixed for all the characters. As before we have
$fM_1+N_1'=N_1, \ fM_2+N_2'=N_2$ where $M_1$, $M_2$ are as in Theorem
\ref{3A} (ii). If $\chi_{(\Lambda_1, \Lambda_2)} $ is parameterized by
$(h_1, h_2, {\rho}_1, {\rho}_2)$ then the above arguments
on removing $f$-hooks applied to
${(\Lambda_1, \Lambda_2)} $ and the symbols corresponding to $({\rho}_1, {\rho}_2)$
show that since $(\Lambda_1, \Lambda_2)$ have $({\pi}_1, {\pi}_2)$ as $f$-cores,
$({\rho}_1, {\rho}_2)$ have $({\sigma}_1, {\sigma}_2 )$ as $f$-cores.
\end{proof}

\noi {\bf Remark }. The above arguments show that the pair $({\rho}_1, {\rho}_2)$ can be regarded
as the $1$-quotient of the pair ${(\Lambda_1, \Lambda_2)} $. This is a special case of
the concept of an $e$-quotient of a symbol in (\cite{O}, Lemma 9).

\noi The case of $H_m$ where $\ell$ divides $q^f+1$ will be considered after
proving Lemma \ref{4D} below, since in that case we have to use
cohooks instead of hooks.

\noi {\bf Remark }. The 4-tuple $(m_1, m_2, {\sigma}_1,
{\sigma}_2)$ (resp. $(h_1, h_2, {\sigma}_1, {\sigma}_2))$ can be
regarded as the ``core'' of the block $B$ (resp. $b$), and the pair
$(M_1, M_2)$ can be regarded as the ``weight" of the block.

 \section{Correspondences between blocks}
\noi All blocks in the rest of the paper will be quadratic unipotent blocks.
The parametrization of blocks described in the last section leads to
 the main theorems of this section. The block correspondences
 that we derive will be between a block $B$
parameterized by $(m_1, m_2, {\sigma}_1, {\sigma}_2, M_1, M_2)$ and a
 block $b$ parameterized by $(h_1, h_2, {\sigma}_1, {\sigma}_2, M_1, M_2)$.
 Suppose $fM_1+N_1'=N_1, \ fM_2+N_2'=N_2$.  If $n$ and $m$ are given by
 $$m_1(m_1+1)/2 + m_2(m_2+1)/2 + 2N_1+2N_2=n, \ {\rm and}$$
 $$h_1(h_1+1) + h_2^2 +N_1+N_2=m$$
 then $B$, $b$ are blocks of $G_n$, $H_m$ respectively.
 For such a fixed pair $(n,m)$  we assume $q>2m$ when
we use the combinatorial description of characters in blocks of
 $Sp(2m,q)$ proved in \cite{BS2}.

\noindent \begin{theorem} \label{4A}  Let $\ell | (q^f-1)$, $f$ odd.
Let $B$ be a quadratic unipotent block of $U(n,q)$ parameterized by
$(m_1, m_2, {\sigma}_1, {\sigma}_2, M_1, M_2)$, where  ${\sigma}_1$
and ${\sigma}_2$ are $f$-cores. Let $b$ be the block of $H_m$
parameterized by $(h_1, h_2,{\sigma}_1, {\sigma}_2, M_1, M_2)$.
Here $n$ and $m$ are as above, and  $(m_1, m_2)$ correspond under
Waldspurger's map to the pair $(h_1, h_2)$. Then
$B$ and $b$ correspond in the sense that (i) their defect groups are
isomorphic, and (ii) assuming $q>2m$, there is a natural bijection between the
quadratic unipotent characters in $B$ and those in $b$.
\end{theorem}

\begin{proof}
Consider the blocks $B$ and $b$ as above.
We use Theorems \ref{3C} and \ref{3D}.
Suppose a character of $U(n,q)$ in $B$ is parameterized
 by $(m_1, m_2, {\rho}_1,{\rho}_2)$. Then the pair $({\rho}_1,{\rho}_2)$
has $f$-core $({\sigma}_1, {\sigma}_2)$. Then the character of $Sp(2m,q)$ parameterized
by $(h_1, h_2, {\rho}_1,{\rho}_2)$ is in $b$. Thus the correspondence between
the quadratic unipotent  characters in $B$ and those in $b$ is given by
 associating the character in $B$ with parameters $(m_1, m_2, {\rho}_1, {\rho}_2)$  with
the character in $b$ with parameters $(h_1, h_2, {\rho}_1, {\rho}_2)$. This  shows (ii).

\noi For (i), let $L$ be the Levi subgroup of the
form $T_1 \times T_2 \times G_{n'}$ as in Theorem \ref{3A} (i).
Then a defect group of $B$ is isomorphic to an $\ell$-Sylow subgroup
of $(T_1 \rtimes ({\bf Z}_{2f} \wr S_{M_1})) \times (T_2\rtimes
({\bf Z}_{2f} \wr S_{M_2}))$ (see \cite{CE}, Theorem 22.9) for the
unipotent block case, which extends to this case). By considering
the Levi subgroup $K$ of $Sp(2m,q)$ again as in Theorem
\ref{3A}, and noting that $\ell$ divides $q^f-1$, we see that the defect
group of $b$ is isomorphic to the defect group of $B$.
\end{proof}

\noi \begin{corollary} \label{4B} The  map $B \rightarrow b$ as above
gives a bijection of the sets

$\lbrace \ell-{\rm blocks \  of}\
U(n,q), \ \ell |(q^f-1)\ (f \ {\rm odd}) ,\ n\geq 0 \rbrace$ and

$ \lbrace \ell-{\rm blocks \ of} \ Sp(2m,q), \ \ell |(q^f-1)\ (f \ {\rm
odd}) ,\ m \geq 0 \rbrace $.

\noi The blocks $B$ and $b$ correspond as in
(i) and (ii) of the theorem.
\end{corollary}

In order to consider the case of $GL(n,q)$  we prove the following
lemma.

\medskip\noi \begin {lemma}\label{4C}  There is a natural bijection between

$\lbrace \ell-{\rm blocks \ of}\ U(n,q), \ \ell |(q^f-1) \ (f \ {\rm
odd})\rbrace $ and

$\lbrace \ell-{\rm blocks \ of}\ GL(n,q), \ \ell
|(q^f+1)\ (f \ {\rm odd})\rbrace$, by Ennola Duality.
\end{lemma}

\begin {proof}  The sets of quadratic unipotent characters of
$GL(n,q)$ and $U(n,q)$ are in bijection via Ennola Duality, such
that the characters in both groups parameterized by the same pair
$({\mu}_1, {\mu}_2)$ correspond. (See e.g. (\cite{BMM}, 3.3) for the
unipotent case, which extends to our case.) By \cite{FS1} both the
$\ell$- blocks  of $GL(n,q)$, $\ell |(q^f+1)$ ($f$ odd) and
  of $U(n,q)$, $\ell |(q^f-1)$ ($f$  odd) are
classified by $2f$-cores.  Thus in both cases the blocks are
parameterized by 6-tuples $(m_1, m_2, {\sigma}_1, {\sigma}_2, M_1,
M_2)$. The map which makes the blocks of $GL(n,q)$ and $U(n,q)$
which are parameterized by the same 6-tuple correspond is then a
bijection, which also induces a bijection of the quadratic unipotent
characters in the blocks.
\end{proof}

\medskip\noi \begin{lemma} \label{4D} There is a natural bijection between
$\ell$ - blocks of $H_n$ where $\ell | (q^f-1)$,  and $\ell$ -
blocks where $\ell | (q^f+1)$, by $f$-twisting. The quadratic unipotent characters
in corresponding blocks also correspond by $f$-twisting.
Here $f$ is odd.
\end{lemma}

\begin{proof} By Lemma \ref{3B}, if a symbol $\Lambda$ is an $f$-core,
then $\what{\Lambda}$ is an $f$-cocore. The $\ell$ - blocks of $H_n$
where $\ell$ divides $q^f-1$ (resp. $q^f+1$) are classified by
$f$-cores (resp. $f$-cocores). If $b$ is an $\ell$-block where
$\ell$ divides $q^{f}-1$ and $b$ corresponds to a pair $(\pi_1,
\pi_2)$ of $f$-cores, let $b^*$ be the $\ell$-block where $\ell$
divides $q^{f}+1$ which corresponds to the pair
$({\what{\pi_1}},{\what{\pi_2}})$ of $f$-cocores.

\noi  The $f$-core (resp. cocore) of a symbol $\Lambda$ is the $f$-twist
of the $f$-cocore (resp. core) of the symbol $\what{\Lambda}$
(\cite{O}, p.235). Thus there is a
bijection between the quadratic unipotent characters in the blocks
$b$ and $b^*$, again by $f$-twisting.
\end{proof}

\noi                We then get the following theorem, analogous to
Theorem \ref{4A}, by Ennola duality and $f$-twisting.

\medskip\noindent \begin{theorem} \label{4E}  Let $\ell$ divide $q^f+1$, $f$ odd.
Let $B$ be a quadratic unipotent block of $GL(n,q)$  and let $B^*$ be the
block of $U(n,q)$ corresponding to $B$ by Lemma \ref{4C}.  Then  consider
the block $b^*$ of $Sp(2m,q)$ corresponding to $B^*$. By Lemma \ref{4D} $b^*$
corresponds, by $f$-twisting  to an $\ell$-block $b$ of $Sp(2m,q)$
where $\ell$ divides $q^f+1$, $f$ odd. Then  $B$ and $b$ correspond
in the sense that (i) their defect groups are isomorphic, and (ii)
assuming $q>2m$,
there is a natural bijection between the quadratic unipotent
characters in $B$ and those in $b$.
\end{theorem}

\noi We now have the following corollary.

\medskip\noi \begin {corollary}\label{4F}

 The above map then gives a bijection of the sets

 $\lbrace
\ell-{\rm blocks \ of}\ GL(n,q), \ell |(q^f+1)\ (f \ {\rm odd}) ,
n\geq 0 \rbrace$ and

$\lbrace \ell-{\rm blocks \ of} \
Sp(2m,q), \ell |(q^f+1)\ (f \ {\rm odd}) , m \geq 0 \rbrace$,

\noi satisfying (i) and (ii) of the theorem.
\end{corollary}

\medskip\noi We now consider the case where  $\ell$ divides $q^f+1$.
where $e=2f$, $f$ even, so that $e \equiv 0 \ ({\rm mod}\ 4)$.

\medskip\noi \begin{theorem} \label{4G} Let $\ell$ divide   $q^f+1$, $f$ even.
Let $B$ be a quadratic unipotent block of $G_n$ parameterized by
$(m_1, m_2, {\sigma}_1, {\sigma}_2, M_1, M_2)$. Then there is a
block $b$ of $H_m$ such that $B$ and $b$ correspond in the sense
that (i) their defect groups are isomorphic, and (ii)
assuming $q>2m$, there is a
natural bijection between the quadratic unipotent characters in $B$
and those in $b$.
\end{theorem}

\begin{proof}  The quadratic unipotent characters in $B$ are the
constituents of $R^{G_n}_L(1 \times {\mathcal E} \times
\chi_{(\lambda_1,\lambda_2)})$, where $L$ is a Levi subgroup of the
form $T_1 \times T_2 \times G_{n'}$, $T_1$ (resp. $T_2$) is a
product of $M_1$ (resp. $M _2$) tori of order $q^{2f}-1$, and $1$
(resp. $\mathcal E$) is the trivial character (resp. character of
order 2) of $T_1$ (resp. $T_2$). Here the pair of partitions
$(\lambda_1,\lambda_2)$ corresponds to $(m_1, m_2, {\sigma}_1,
{\sigma}_2)$ where $({\sigma}_1, {\sigma}_2)$ are $f$-cores, and we
have a character $\chi_{(\pi_1, \pi_2)}$
of a group $H_{m'}$ corresponding to $(h_1, h_2,
{\sigma}_1, {\sigma}_2)$. By the proof of Theorem \ref{3D}, $(\pi_1, \pi_2)$
are $f$-cores since $({\sigma}_1, {\sigma}_2)$ are $f$-cores.
The character obtained from $\chi_{(\pi_1, \pi_2)}$ by
$f$-twisting is of the form $\chi_{(\tau_1, \tau_2)}$, where the
symbols $\tau_1, \tau_2$ are $f$-cocores. Let $b$ be the
$\ell$-block of a group $H_m$ corresponding to this character and $M_1, M_2$,
i.e. the block $b$ such that the  quadratic unipotent
characters  in it are constituents of $R^{H_m}_K(1 \times {\mathcal
E} \times \chi_{(\tau_1,\tau_2)})$, where $K$ is a Levi subgroup of
the form $T_1 \times T_2 \times H_{m'}$, $T_1$ (resp. $T_2$) is a
product of $M_1$ (resp. $M _2$) tori of order $q^{f}+1$, and $1$
(resp. $\mathcal E$) is the trivial character (resp. character of
order 2) of $T_1$ (resp. $T_2$) (Theorem \ref{3A}(iii)).
Then $B$ and $b$ correspond as required: For (i) the proof is as in
Theorem \ref{4A}. For (ii) we note that there is a bijection by
$f$-twisting between the quadratic unipotent
constituents of $R^{H_m}_K(1 \times {\mathcal
E} \times \chi_{(\tau_1,\tau_2)})$ and those of $R^{H_m}_K(1 \times {\mathcal
E} \times \chi_{(\pi_1, \pi_2)}$ (\cite{O}, p.235). However,
the  quadratic unipotent constituents of the latter are in bijection
with the quadratic unipotent characters in $B$, since $({\sigma}_1, {\sigma}_2)$
are the $2$-quotients of $(\lambda_1,\lambda_2)$. This proves the result.
\end{proof}

\noi Summarizing, we have bijections between the following:

(i) $\lbrace \ell-{\rm blocks \ of}\ U(n,q),\ \ell |(q^f-1)\ (f \ {\rm
odd}) , n\geq 0 \rbrace $ and

$\lbrace \ell-{\rm blocks
\ of} \ Sp(2m,q), \ \ell |(q^f-1)\ (f \ {\rm odd}) , m \geq 0
\rbrace$.

(ii) $\lbrace \ell-{\rm blocks \ of}\ GL(n,q),\ \ell |(q^f+1)\
(f \ {\rm odd}) , n\geq 0 \rbrace $ and

$\lbrace \ell-{\rm blocks
\ of} \ Sp(2m,q), \ \ell |(q^f+1)\ (f \ {\rm odd}) , m \geq 0
\rbrace$.

(iii)
$\lbrace \ell-{\rm blocks \ of}\ U(n,q),\ \ell |(q^f+1)\ (f \ {\rm
even}) , n\geq 0 \rbrace $ and

$\lbrace \ell-{\rm blocks
\ of} \ Sp(2m,q), \ \ell |(q^f+1)\ (f \ {\rm even}) , m \geq 0
\rbrace$.

(iv)
$\lbrace \ell-{\rm blocks \ of}\ GL(n,q),\ \ell |(q^f+1)\ (f \ {\rm
even}) , n\geq 0 \rbrace $ and

$\lbrace \ell-{\rm blocks
\ of} \ Sp(2m,q), \ \ell |(q^f+1)\ (f \ {\rm even}) , m \geq 0 \rbrace$.

\section{Perfect Isometries}

\noi In this section we consider blocks $B$ and $b$
of a pair $G_n$ and $H_m$ which correspond as in Section $4$.
In that case we assume that $\ell$ does not divide the order
of the Weyl groups of $G_n$ and $H_m$, and thus that $\ell$
is large in the sense of (\cite{BMM}, 5.1). This implies that
the blocks considered have abelian defect groups.

\noindent
We generalize the result on perfect isometries
between unipotent blocks of \cite{BMM} to quadratic unipotent blocks.
We use the classification of blocks by $e$-cuspidal pairs
and the description of characters in the blocks
(\cite{CE}, 22.9; \cite{BS}, 3.9, \cite{BS2}, Section 7).

\noi We first describe the defect groups and their
normalizers  of the blocks under consideration (\cite{BMM}, pp.46,50).

\noi
Case 1. $G=G_n$. Let $B$ be a block of $G$ as in Section 3, so that
$\ell$ divides $q^{2f}-1$. Let $L$ be a Levi subgroup of the form
$T_1 \times T_2 \times G_r$, where $T_1$ (resp. $T_2$) is a
product of $M_1$ (resp. $M_2$) tori of order $q^{2f}-1$. The defect
group of $B$ is then  a Sylow $\ell$-subgroup of $T_1 \times T_2$.

\noi Case 2. $G=H_n$. Let $b$ be a block of $G$ as in Section 3,
so that $\ell$ divides $q^{f}-1$ or $q^{f}+1$ . Let $L$ be a Levi
subgroup of the form $T_1 \times T_2 \times H_r$, where $T_1$
(resp. $T_2$) is a product of $M_1$ (resp. $M_2$) tori of order
$q^{f}-1$ or $q^{f}+1$. The defect group of $b$ is then a Sylow
$\ell$-subgroup of $T_1 \times T_2$.

\noi We note that the defect groups of  two blocks $B$ and $b$ which
correspond as in Section 4 are isomorphic.

\noi In each case, we have $W_G(L)=N_G(L)/L \cong {\bf Z}_{2f} \wr
S_{M_1+M_2}$, where $S_N$ is the symmetric group of degree $N$. Now
suppose $\lambda$ is a quadratic unipotent $2f$-cuspidal character
(resp. $f$-cuspidal character) of $G_r$ (resp. $H_r$).  Then
we have in each case $W_G(L, \lambda)=N_G(L, \lambda)/L = W_1 \times W_2$
where $ W_1 \cong {\bf Z}_{2f} \wr
S_{M_1}$ and $W_2 \cong {\bf Z}_{2f} \wr S_{M_2}$.

\noindent The results of Brou\'e, Malle and Michel (\cite {BMM}, 3.2,
5.15) can be modified as follows.

\noi \begin {theorem} \label{5A}  Let $G=G_n$
or $H_n$ and $L$ a Levi subgroup of $G$  as in Case 1 or Case 2 above.
 Let $\lambda$ be a quadratic unipotent
 character of $L$ of the form $1 \times
{\mathcal E} \times \chi$, where $1$ is a trivial character
(resp. character of order 2) of $T_1$ (resp. $T_2$), and
${\chi}$ is in a block of defect $0$ of $G_r$ or $H_r$, so that
 $(L,\lambda)$ is an $e$-cuspidal pair in Case 1 and
 an $f$-cuspidal pair in Case 2.

\noi Let $M$ be an $2f$-split Levi subgroup containing $L$ in Case 1
 or an $f$-split or $2f$-split Levi subgroup containing $L$ in Case 2.
 We then have an isometry  $I^M_{(L, \lambda)}$ between
 the ${\bf Z}$-spans of the set ${\rm Irr}(W_M(L,\lambda))$ and of
 the set of constituents of $R_L^M(\lambda)$,   such that
 $ R_M^G \ . \ I^M_{(L, \lambda)} = I^G_{(L, \lambda)}\ . \ {\rm Ind}^{W_G(L,
 \lambda)}_{W_M(L, \lambda)}$.
\end{theorem}

\noi  \begin{proof}
 If $G=G_n$ (resp. $H_n$) the quadratic unipotent characters are of the form
$\chi_{({\mu}_1, \mu_2)}$ (resp. $\chi_{(\Lambda_1,\Lambda_2)}$)
where $\mu_1, \mu_2$ are partitions and $\Lambda_1,\Lambda_2$ are
symbols. In this case the characters are in a fixed Lusztig series
and thus in bijection with the unipotent characters of the centralizer
of a semisimple element. Thus we have fixed integers $n_1, n_2$ such that
$n_1+ n_2=n$, and $\mu_1, \mu_2$ are partitions of $n_1, n_2$ respectively
and $\Lambda_1,\Lambda_2$ are symbols of rank $n_1, n_2$ respectively.

\noi In the case of the unipotent characters of $G_n$ and $H_n$
the group $M$ has been described in (\cite {BMM}. p.46, p.49-52).
From our choice of $f$ the group $M$ in our case can be assumed
to have the following form: $M = GL(b,q^{2f}) \times G_k$  for some $b,k$
in the case of $G_n$ and  $M = GL(b,q^{f}) \times H_k$
 or $M = U(b,q^{f}) \times H_k$ for some $b,k$ n the case of $H_n$,
 We have $b \leq M_1+M_2$.

\noi Suppose $L$ is embedded in $M$ as
follows. Let $T_1=T_{1,1}\times T_{1,2}$, $T_2=T_{2,1}\times T_{2,2}$.

\noi Case 1. Let $T_{1,1}\times T_{2,1}\subseteq GL(b,q^{f})$, $T_{1,2}\times
T_{2,2}\times G_r \subseteq G_k$,
where $T_{1,1}$ (resp. $T_{2,1}$) is isomorphic to $b_1$ (resp.
$b_2$) copies of tori of orders $q^{2f}-1$.

\noi Case 2. Let $T_{1,1}\times T_{2,1}\subseteq GL(b,q^{f}) \ {\rm or} \ U(b,q^{f})$,
$T_{1,2}\times T_{2,2}\times H_r \subseteq H_k$,
where $T_{1,1}$ (resp. $T_{2,1}$)  is isomorphic to $b_1$ (resp.
$b_2$) copies of tori of orders $q^{f}-1$ or $q^{f}+1$.

\noi Recall that $W_G(L, \lambda)=N_G(L, \lambda)/L = W_1 \times W_2
 \cong {\bf Z}_{2f} \wr
S_{M_1} \times {\bf Z}_{2f} \wr S_{M_2}$.

\noi
Since the character $\lambda$ takes the value $1$
on $T_1$ and $\mathcal E$ on $T_2$, we see that for both $G_n$
and $H_n$ we get $W_M(L,\lambda) = {W_1}' \times {W_2}'$ where ${W_1}' \cong
S_{b_1} \times ({\bf Z}_{2f})\wr S_{M_1-b_1}$ and ${W_2}'
\cong S_{b_2} \times ({\bf Z}_{2f})\wr S_{M_2-b_2}$.

\noi Now we consider Lusztig induction $R_L^M(\lambda)$ where
$\lambda$ is of the form $\chi_{(\lambda_1, \lambda_2)}$,
where $(\lambda_1, \lambda_2)$ is a pair of
partitions or symbols. Using results of Waldspurger \cite{W1}
it was shown in (\cite{BS2}, 4.2) for the case of $H_n$ that
Lusztig induction commutes with Jordan decomposition. More precisely,
we have:  The constituents of $R_L^M(\lambda)$
are of the form $\chi_{(\mu_1,\mu_2)}$, where the $\mu_i$ are
obtained from the $\lambda_i$ by adding a succession of hooks
or cohooks.

\noi We consider the case of $H_n$. Then $C_{G^{*}}(s) = K_1 \times K_2$
where $K_1$ (resp. $K_2$) is isomorphic to $SO(2m_1+1)$ (resp. $O^{\pm 1}(2m_2)$)
 with $m_1+m_2=n$. We have subgroups $M^{*}$, $L^{*}$ which are
intersections of subgroups dual to $M$, $L$ with $C_{G^{*}}(s)$. Then we have
 $M^{*}= M_1 \times M_2$, $L^{*} = L_1 \times L_2$, where $M_1, L_1 \subseteq K_1$
 and $M_2, L_2 \subseteq K_2$, and characters $\lambda_i$ of $L_i$, $i=1,2$.
 By applying (\cite{BMM}, 3.2) to the  $K_i$ we have isometries between
  the ${\bf Z}$-spans of the set ${\rm Irr}(W_{M_i}(L_i,\lambda_i))$ and
 the set of constituents of $R_{L_i}^{M_i}(\lambda_i)$   such that
 $ R_{M_i}^{K_i} \ . \ I^{M_i}_{L_i, (\lambda_i)} =
 I^{K_i}_{(L_i, \lambda_i)}\ . \ {\rm Ind}^{W_{K_i}(L_i,
 \lambda_i)}_{W_{M_i}(L_i,\lambda_i)}$, $i-1,2$. Here we note that in the case
 of groups of the form $O^{\pm 1}(2m_2)$ we use the results of Malle \cite{M} which
 extend (\cite{BMM}, 3.2) to disconnected groups. We also use Lusztig induction
 in disconnected groups (see \cite{DM}).

\noi We  define $I^M_{(L, \lambda)}$ as follows. Let $(\psi_1, \psi_2) \in
{\rm Irr}(W_M(L,\lambda) = {W_1}' \times {W_2}')$. We identify ${W_i}'$ with
$W_{M_i}(L_i,\lambda_i)$. Suppose $I^{K_i}_{(L_i, \lambda_i)}(\psi_i)=
\chi_{\mu_i}$, a constituent of $R_{L_i}^{M_i}(\lambda_i)$. Then define
$I^M_{(L, \lambda)}((\psi_1, \psi_2))= \chi_{(\mu_1,\mu_2)}$.
 We then have an isometry $I^M_{(L, \lambda)}$ between the ${\bf Z}$-spans of
 the set ${\rm Irr}(W_M(L,\lambda))$ and the set of constituents of
 $R_L^M(\lambda)$   such that
 $ R_M^G \ . \ I^M_{(L, \lambda)} = I^G_{(L, \lambda)}\ . \ {\rm Ind}^{W_G(L,
 \lambda)}_{W_M(L, \lambda)}$.

\noi The case of $G_n$ is similar and easier. It was shown in \cite{FS1} that
Lusztig induction commutes with Jordan decomposition in that case.
This proves the theorem.
\end{proof}

\noi The proof of Theorem \ref{5A} is a formal extension of (\cite{BMM},
3.2). We now give an explicit description of the maps $I^G_{(L,
\lambda)}$ in our case, as in (\cite{BMM}, pp.47,50).

\noi In the case of $G=G_n$, the parametrization of quadratic
unipotent characters  is either by pairs of partitions $\mu_1, \mu_1$ or by
$4$-tuples $(m_1, m_2, \rho_1, \rho_2)$. In the case of $U(n.q)$,
the latter arises from their construction
by Lusztig by Harish-Chandra induction. Consider the characters occurring in
$R_L^G(\lambda)$ for appropriate $( L, \lambda)$. The description given in
(\cite{BMM}, p.50) shows that, given such a character, each $\mu_i$
corresponds to  a $2f$-tuple of partitions whose sizes add up to
$M_i$,  $i=1,2$. (Here the $M_i$ are weights, denoted by $a$ in op.cit.
where the characters are unipotent.) These $2f$-tuples are in fact $2f$-quotients of the
$\mu_i$. Now Olsson (\cite{O}, p.233) has defined the $e$- quotient of a symbol for a
positive integer $e$, and his definition shows that the $2f$-quotients  of the $\mu_i$
are in fact the $2f$-quotients  of the $\rho_i$.
Since the irreducible characters of
$W_G(L, \lambda)$ are parameterized by pairs of $2f$-tuple of partitions,
this defines the map $I^G_{(L, \lambda)}$ in this case.

\noi Now consider the case of $G=H_n$ where the parametrization of quadratic
unipotent characters  is either by pairs of symbols $\Lambda_1, \Lambda_2$ or by
$4$-tuples $(h_1, h_2, \rho_1, \rho_2)$. Here again the latter arises from their construction
by Lusztig \cite{L1} by Harish-Chandra induction.
The connection between the pairs
$\Lambda_1, \Lambda_2$ and the pairs $(\rho_1, \rho_2)$ was stated in the proof of
Lemma 2.2. In (\cite{BMM}, p.50) it is shown how the map $I^G_{(L, \lambda)}$
is defined for unipotent characters in this case.
Using this we define the map $I^G_{(L, \lambda)}$ by taking $2f$-quotients of the $\rho_i$.

\noi Then we have a bijection with signs between the set of quadratic unipotent characters
occurring in $R_L^G(\lambda)$ and the set ${\rm Irr}(W_G(L,\lambda))$. We then see that
the character of $G_n$ parameterized by $(m_1, m_2, {\rho}_1, {\rho}_2)$ and the
character of $H_n$ parameterized by $(h_1, h_2, {\rho}_1, {\rho}_2)$ correspond to the
same character in ${\rm Irr}(W_G(L,\lambda))$ in the above bijection, where we
choose $G, L, \lambda$ appropriately in each case.

\noi We thus have:

\begin {theorem} \label{5B} Let $B$ and $b$ be blocks with
abelian defect groups of a pair $G_n$
and $H_m$ which correspond as in Section $4$, Theorems \ref{4A}, \ref{4E}, \ref{4G}.
Then the correspondence
between the sets of quadratic unipotent characters in $B$ and $b$
factors through the isometry of these sets with the sets
${\rm Irr}(W_G(L,\lambda))$ with appropriate $G, L, \lambda$ for $G_n$
and $H_m$.
\end{theorem}

\noi Next we consider perfect isometries, and an analog of (\cite{BMM},
5.15). For this we need to consider characters $\theta \in {\rm
 Irr}(Z(L)_{\ell})$ for $L$ a Levi subgroup of $G=G_n$ or $G=H_n$ as in
 Theorem \ref{3A} (in the case of $H_n$ this subgroup was denoted by $K$).
 In (\cite{BMM}, 5.15)  a subgroup $G(\theta)$ of
 $G$ has been introduced. Here we give an alternative definition of
 this group, analogous to a definition in (\cite{CE1}, p.163).
 Let $L^*$ be a subgroup of $G^*$ in duality with $L$,  and
  $t \in (Z(L^*)_{\ell})$ an $\ell$-element. Then $C_{G^*}(t)^0$ is a Levi
 subgroup of $G^*$ and there is a subgroup $G(t)$ of $G$ in duality
 with $C_{G^*}(t)^0$. Since $\ell$ is odd $G(t)$ is isomorphic to $G(\theta)$, where
 $\theta$ corresponds to  a linear character
$\hat t$ of $G(t)$, defined when we have chosen a fixed embedding of
${\overline{\bf F}_q^*}$ into ${\overline{\bf Q}_l}$. We
 will use the subgroup $G(t)$ instead of $G(\theta)$ in the
 following. The groups $G(t)$ can be explicitly described as
 being isomorphic to $\prod_iGL(m_i, q^{2f}) \times G_r$ or
$\prod_iU(m_i, q^{2f}) \times G_r$ in the case of $G_n$, and to
$\prod_iGL(m_i, q^{f}) \times H_r$ or $\prod_iU(m_i, q^{f}) \times
H_r$ in the case of $H_n$.

\noindent We consider a quadratic unipotent block $b$ of $G=G_n$ or $H_n$.
The quadratic unipotent
characters in $b$ are constituents of $R^G_L(1 \times {\mathcal
E} \times {\chi}_{({\pi_1},{\pi_2})})$, where $L$ is a suitable Levi
subgroup and $({\pi_1},{\pi_2})$ are $2f$-core partitions or
$f$-core or $f$-cocore symbols. We now
consider the other characters in $b$. We apply (\cite{CE1}, Theorem 2.8)
which describes all the constituents in $b$ with only the restriction
that $\ell$ is good, which is true in our case. We also note that since
$t$ is an $\ell$-element, $G(t)$ is connected and $R_{G(t)}^G$ is an isometry.
Then we get that a character in $b$ is of the form $R_{G(t)}^G({\hat
t}\chi)$, up to sign, where $\chi$ is a quadratic unipotent character of $G(t)$.
We also note that an irreducible character of  $Z(L)_{\ell}\rtimes W_G(L, \lambda)$ can
be written as ${\hat t}\tau$ for some $t \in Z(L^*)_{\ell}$ and
an irreducible character  $\tau$ of $W_G(L, \lambda)$ as in (\cite{BMM}, p.71).

\noi The map $R_{G(t)}^G$ in the theorem of Cabanes-Enguehard (op. cit.) involves a
parabolic subgroup. By a recent result of Bonnaf\'e-Michel [J.Algebra 327 (2011), 506-526)]
showing that if $q > 2$ Mackey's Theorem holds, Lusztig induction $R_L^G$ where $G$
is a reductive group and $L$ is a Levi subgroup is independent of the choice of a
parabolic subgroup containing $L$.

\noi We now state the analog of (\cite{BMM}, 5.15) in our case.

\noindent \begin {theorem} \label{5C} Let $G=G_n$ or $G=H_n$.The map
$$I^G_{(L, \lambda)}: {\bf Z}{\rm Irr}(Z(L)_{\ell}\rtimes W_G(L, \lambda)) \rightarrow
{\bf Z}\ {\rm Irr}(G,b)$$ such that
$${\rm Ind}_{{\bf Z}{\rm Irr}(Z(L)_{\ell}.W_{G(t)}(L, \lambda)}^{{\bf Z}{\rm Irr}(Z(L)_{\ell}.W_G(L, \lambda)}({\hat t}\tau) \rightarrow
R_{G(t)}^G({\hat t}I^{G(t)}_{(L, \lambda)}(\tau))$$ is an
$\ell$-perfect isometry between  $(Z(L)_{\ell}\rtimes W_G(L,
\lambda)),b( 1 . (1 \times {\mathcal E})))$ and $(G, b)$.
\end{theorem}

\noi Here we interpret the character $1 . (1 \times {\mathcal E})$ as
follows. We have $W_G(L, \lambda)= W_1 \times W_2$ as in Theorem \ref{5A}.
 We take the trivial character $1$ on
$Z(L)_{\ell}$, the character $1$ on $W_1$ and the character
$\mathcal E$ on $W_2$. Then $b( 1 . (1 \times {\mathcal E})))$ is the
block containing $1 . (1 \times {\mathcal E})$ of $(Z(L)_{\ell}\rtimes
W_G(L, \lambda))$.

\noi \begin{proof}   We use the definition of $\ell$-perfect
isometry given in (\cite{BMM}, \ 5.11). We note the following points in
the proof of (\cite{BMM}, \ 5.15) at which unipotent characters have to
be replaced by quadratic unipotent characters.

\begin{itemize}
\item  The $f$-Harish-Chandra theory was proved for quadratic
unipotent characters in classical groups in (\cite{BS}), which gives
us the analog of (\cite{BMM},\ 5.19, \ 5.18).

\item We have verified the extension to our case of
(\cite{BMM}), \ 3.2) in Theorem \ref{5A}. This is used in (\cite{BMM},\ 5.17).

\item An $e$-cuspidal or $f$-cuspidal quadratic unipotent character is of defect
$0$ for $G=G_n$ or $G=H_n$. This follows by Jordan decomposition and
by degree considerations.  This generalizes (\cite{BMM}, \ 5.21).

\end{itemize}

\noi Then the proof is formally completely analogous to that of
(\cite{BMM}, \ 5.15). Part (ii) of the result shows that there is an
 $\ell$-perfect isometry between
$(Z(L)_{\ell}\rtimes W_G(L, \lambda), 1 . (1 \times {\mathcal
E}))$ and $(G, b)$.
\end{proof}

\noi We now consider the groups $G_n$ and $H_n$.

\noindent \begin {theorem}\label{5D} We have $\ell$-perfect isometries
in the sense of (\cite{BMM}, 5.11)
between corresponding blocks of the following groups:

(i) $ GL(n,q), \ \ell |(q^f+1)\ (f \ {\rm
odd})$ and $ Sp(2m,q), \ \ell |(q^f+1)\ (f \ {\rm odd})$,

(ii) $ U(n,q),\ \ell |(q^f-1)\ (f \ {\rm
odd}) $ and $ Sp(2m,q), \ \ell |(q^f- 1)\
(f \ {\rm odd})$ .

(iii) $ GL(n,q), \ \ell |(q^f+1)\ (f \
{\rm even})$ and $ Sp(2m,q), \ \ell
|(q^f+1)\ (f \ {\rm even}) $.

(iv) $ U(n,q), \ \ell |(q^f+1)\ (f \ {\rm
even})$  and $ Sp(2m,q),\ \ell |(q^f+ 1)\
(f \ {\rm even})$ .

\noi In (i) and (ii), the block of $G_n$ parameterized by  $(m_1, m_2,
{\sigma}_1, {\sigma}_2, M_1, M_2)$, where  $m_1(m_1+1)/2 +
m_2(m_2+1)/2 +2N_1+2N_2=n,$ corresponds to the block of $H_m$ parameterized
by $(h_1, h_2, {\sigma}_1, {\sigma}_2, M_1, M_2)$, where $h_1(h_1+1)
+ h_2^2 +N_1+N_2=m$. For the connection between the $M_i$ and the $N_i$
see Theorem \ref{3C}. In cases (iii) and (iv) the blocks correspond as
in Theorem \ref{4G}.

\end{theorem}

\noi \begin{proof} The theorem follows from Theorem \ref{5C}, since
in each case there is a perfect isometry between the blocks in
question and a block of a ``local" group of the form
$Z(L)_{\ell}\rtimes W_G(L, \lambda)$.
\end{proof}

\begin{theorem}\label{5E}  Suppose a block $B$ of $G_n$ and a block $b$ of $H_n$
correspond as in Theorem \ref{5D}. The quadratic unipotent characters in $B$ and $b$
correspond under the isometry as follows:  In cases (i) and (ii) above, the
character of $G_n$ parameterized by $(m_1, m_2, {\rho}_1, {\rho}_2)$ corresponds to
the character of $H_n$ parameterized by $(h_1, h_2, {\rho}_1, {\rho}_2)$.
In cases (iii) and (iv) the characters correspond as
in Theorem \ref{4G}.
\end{theorem}

\begin{proof}
The theorem follows from the fact that
in the map $I^G_{(L, \lambda)}$ in Theorem \ref{5C} we can take $t=1$.
Using Theorem \ref{5B} we get the correspondence between characters
as in Theorem \ref{2D}.
\end{proof}

\noi {\bf Remark}.  The case of $G_n$ is easier than that of $H_n$, as is seen below.

Let $G=G_n$, $B$ a quadratic unipotent block of $G_n$. The quadratic
unipotent
 characters in $B$ are of the form $\chi_{({\mu}_1, \mu_2)}$ in the
 Lusztig series  ${\mathcal E}(G, (s))$, where
$({\mu}_1, \mu_2)$ are partitions of a fixed pair $k_1, k_2$
respectively. By a result of Bonnaf\'e and Rouquier the block $B$ is
Morita equivalent to a block $B(s)$ of $C_G(s)$. Now since $s$ is
central in $C_G(s)$ the block $B(s)$ can be regarded as the product
of two unipotent blocks of $C_G(s)$, and thus (\cite{BMM},5.15) can
be applied to it. We get a perfect isometry between the block and a
quadratic unipotent block of the ``local subgroup"
$Z(L)_{\ell}\rtimes W_G(L, \lambda)$.

\noi We now consider signs appearing in the perfect isometries of Theorem \ref{5D}
and Theorem \ref{5E}.
Consider a  quadratic unipotent character $\chi$ of $G_n$
 parameterized by  a pair $(\lambda_1, \lambda_2)$ of partitions which
corresponds to the
 quadratic unipotent character $\psi$ of $H_m$ parameterized by
 a pair $(\Lambda_1, \Lambda_2)$ of symbols
 under the perfect isometry. Enguehard ({\cite{E2}, p.34) has used the combinatorics
 of partitions and symbols to  define a sign $\nu_e$
 on partitions and symbols and uses them to calculate
  the signs which appear in ( {\cite{E2}, Theorem B), which is the same theorem
 as (\cite{BMM}, 3.2). Thus the sign appearing in the correspondence between
 $\chi$ and $\psi$ as above is $\nu_e(\lambda_1)\nu_e(\lambda_2) \nu_e(\Lambda_1)\nu_e(\Lambda_2)$.

\section {Endoscopic groups}

Let $G$ be a finite reductive group, $\ell$ a prime as before, and
$(s)$ an $\ell$-prime semisimple class in $G^*$. Let $B$ be an
$\ell$-block of $G$ parameterized by $(s)$. M.Enguehard has proved
the following \cite{E1}. There is a (possibly disconnected) group
$G(s)$ which need not be a subgroup of $G$, and a block $B(s)$ of
$G(s)$ such that $B$ and $B(s)$ correspond, in the following sense:

\begin{itemize}
\item There is a bijection between characters in $B$ and $B(s)$
\item The defect groups of $B$ and $B(s)$ are isomorphic
\item The Brauer categories of $B$ and $B(s)$ are equivalent
\end{itemize}

The group $G(s)$ is dual to the centralizer of $s$ in $G^*$.
 We call $G(s)$ an endoscopic group of $G$, in analogy with a
terminology used in $p$-adic groups. We describe the endoscopic
groups in our case (\cite{E1}, 3.5.4).

Case 1. $G=G_n$, $B$ corresponds to the Levi subgroup $L$  of the
form $T_1 \times T_2 \times G_{n'}$, $T_1$ (resp. $T_2$) is a
product of $M_1$ (resp. $M _2$) tori of order $q^{2f}-1$, and we
take a character of $L$ to be $1$ (resp. $\mathcal E$)  on $T_1$
(resp. $T_2$) and the character ${\chi}_{(\lambda_1,\lambda_2)}$
 of defect $0$ of $G_{n'}$. The pair $(\lambda_1,\lambda_2)$
 corresponds to a pair $(m_1, m_2)$ as before. Then $s \in G_n=G_n^*$
 has $n_1$ (resp. $n_2$) eigenvalues 1 (resp. -1) where
 $n_1=2fM_1+ |\lambda_1|$, $n_2=2fM_2+ |\lambda_2|$.

 Then $G(s) = G_n(s)\cong G_{n_1} \times G_{n_2}$.

 Case 2. $G=H_m$, $B$ corresponds to the Levi subgroup $L$  of the
form $T_1 \times T_2 \times H_{m'}$, $T_1$ (resp. $T_2$) is a
product of $M_1$ (resp. $M _2$) tori of order $q^{f}-1$ or
$q^{f}+1$, and we take a character of $L$ to be $1$ (resp. $\mathcal
E$) on $T_1$ (resp. $T_2$) and the character
${\chi}_{(\pi_1,\pi_2)}$
 of defect $0$ of $H_{m'}$. The pair $(\pi_1,\pi_2)$
 corresponds to a pair $(h_1, h_2)$ as before. Then $s \in H_m^*$
has $k_1$ (resp. $k_2$) eigenvalues 1 (resp. -1) where
 $k_1=fM_1+ \ {\rm rank}\pi_1$,  $k_2=fM_2+ \ {\rm rank}\pi_2$. We note that
$H_m^* \cong SO(2m+1)$.

Then $H(s) = H_m(s) \cong Sp(2k_1,q) \times O(2k_2,q)$. Here we get
$O^{+}(2k_2,q)$ if $h_2$ is even and $O^{-}(2k_2,q)$ if $h_2$ is odd
(see \cite{W1}, 4.3).

Under the Jordan decomposition of characters, the quadratic unipotent
characters of $G_n$ and $H_m$ correspond to characters of $G_n(s)$ and
$H_m(s)$ respectively which are tensor products of unipotent characters
with a fixed linear character $\hat{s}$. There is a bijection between the set of
quadratic unipotent blocks of $G_n$  (resp. $H_m$) and the set of blocks of
$G_n(s)$  (resp. $H_m(s)$ ) which contain the characters as above, and then
a bijection between the set of
quadratic unipotent blocks of $G_n$  (resp. $H_m$) and the set of unipotent
blocks of $G_n(s)$  (resp. $H_m(s)$ ) .
The proof of the theorem below follows from these bijections.

\noi \begin{theorem} \label{6A} We have block correspondences between
 unipotent blocks of endoscopic groups as follows. As in
Theorems \ref{4A},\ref{4E},\ref{4G} we have (i) the defect groups of
corresponding blocks $B$ and $b$ are isomorphic, and (ii) there is a
natural bijection between the  unipotent characters in $B$ and those
in $b$.

$\lbrace \ell-{\rm blocks \ of}\ GL(n_1,q)\times GL(n_2,q), \ \ell
|(q^f+1)\ (f \ {\rm odd}) ,\ n\geq 0 \rbrace$ and
   $ \lbrace \ell-{\rm blocks
\ of} \ Sp(2k_1,q)\times  O(2k_2,q), \ \ell |(q^f+1)\ (f \ {\rm odd}) ,
\ m \geq 0 \rbrace$.

$\lbrace \ell-{\rm blocks \ of}\ U(n_1,q)\times U(n_2,q), \ \ell
|(q^f-1)\ \ (f \ {\rm odd}) ,\ \ n\geq 0 \rbrace$ and
   $ \lbrace \ell-{\rm blocks
\ of} \ Sp(2k_1,q)\times  O(2k_2,q),  \ell |(q^f-1)\ (f \ {\rm odd}) ,
\ m \geq 0 \rbrace$

$\lbrace \ell-{\rm blocks \ of}\ U(n_1,q)\times U(n_2,q), \ \ell
|(q^f+1)\ \ (f \ {\rm even}) ,\ \ n\geq 0 \rbrace $ and
   $ \lbrace \ell-{\rm blocks
\ of}\ Sp(2k_1,q)\times  O(2k_2,q), \ \ell |(q^f+1)\ (f \ {\rm even}) , m
\geq 0 \rbrace$

$\lbrace \ell-{\rm blocks \ of}\ GL(n_1,q)\times GL(n_2,q), \ \ell
|(q^f+1)\ (f \ {\rm even}) , \ n\geq 0 \rbrace $ and
   $ \lbrace \ell-{\rm blocks
\ of} \ Sp(2k_1,q)\times  O(2k_2,q), \ \ell |(q^f+1)\ (f \ {\rm even}) ,
\ m \geq 0 \rbrace$

\noi Here $n=n_1+n_2$ and $m=k_1+k_2$ correspond as before, and
$n_1$, $n_2$, $k_1$, $k_2$ are as defined.
\end{theorem}

\noindent We now consider
 perfect isometries between the corresponding blocks above,
which follow easily from the case of \cite{BMM}.

\noi Let $B(s)$ be an $\ell$-block of $G_n(s)=G_1 \times G_2$,
where $G_1=G_{n_1}$ and $G_2=G_{n_2}$. Then
$B(s)$ factorizes as $B_1(s) \times B_2(s)$ where $B_1(s)$ and
$B_2(s)$ are blocks of $G_1$ and $G_2$ respectively. There are Levi
subgroups $L_1(s)$ and $L_2(s)$ of $G_1$ and $G_2$ respectively such
that $L_1(s)= T_1 \times G_{n'_1} $ and $L_2(s)= T_2 \times G_{n'_2}
$. Here $T_1$ (resp. $T_2$) is a product of $M_1$ (resp. $M_2$) tori
of order $q^{2f}-1$. Consider the ``local group" $({(T_1)}_{\ell}
\times  {(T_2)}_{\ell} )\ltimes ({\bf Z}_{2f} \wr S_{M_1} \times {\bf
Z}_{2f} \wr S_{M_2})$. A character $\theta$ of ${(T_1)}_{\ell}
\times  {(T_2)}_{\ell} $ factorizes as ${\theta}_1 \times
{\theta}_2$, where ${\theta}_i \in {\rm Irr}({(T_i)}_{\ell})$,
$i=1,2$. Then the pair ${\theta}_1 , {\theta}_2$ determines a pair
$(t_1,t_2)$ of $\ell$-elements in $G_1 \times G_2$, and then a
subgroup $G(t_1) \times G(t_2)$ of $G_1 \times G_2$ which plays a
role analogous to that of $G(t)$ in the case of $G_n$. Since $B(s)$
is a product of blocks of $G_1$ and $G_2$ containing characters
which are products of a fixed linear character and {\em unipotent}
characters, by an
application of \cite{BMM} we get a perfect isometry of $(G_n(s),B(s))$
with the principal block of the ``local group" $({(T_1)}_{\ell}
\times  {(T_2)}_{\ell} )\rtimes ({\bf Z}_{2f} \wr S_{M_1} \times {\bf
Z}_{2f} \wr S_{M_2})$.

\noi In the case of $(H_m(s),b(s))$ , similarly we get a perfect
isometry  with the principal block of the same ``local group"
${(T_1)}_{\ell} \times {(T_2)}_{\ell} \rtimes ({\bf Z}_{2f} \wr
S_{M_1} \times {\bf Z}_{2f} \wr S_{M_2})$. We note that here the
elements $t_1,t_2$ are to be taken in the dual group ${H_m}^*$. We
also note that as before, in the case where we have a group of the form
$O(2k,q)$ we use results of Malle \cite{M} extending
\cite{BMM} to disconnected groups. Finally we get a perfect isometry
between $B(s)$ and $b(s)$.

\end{document}